\documentclass[12pt]{article}
\usepackage[latin1]{inputenc}
\usepackage[british]{babel}
\usepackage{lmodern}
\usepackage[T1]{fontenc}
\usepackage[paper=a4paper, left=28mm, right=25mm, top=29mm, bottom=25mm]{geometry}
\usepackage{latexsym,amsfonts,amsmath,graphics}
\usepackage{epsfig}
\usepackage{enumitem}
\usepackage{amssymb,mathrsfs}
\usepackage{verbatim}

\newtheorem{theorem}{Theorem}
\newtheorem{lemma}{Lemma}
\newtheorem{corollary}{Corollary}

\newtheorem{conjecture}{Conjecture}

\newtheorem{remark}{Remark}
\newenvironment{proof}{\begin{trivlist} \item[\hskip\labelsep{\it Proof.}]}{$\hfill\Box$\end{trivlist}}

\newcommand{\rd}{\,\mathrm{d}}

\newcommand{\bsa}{\boldsymbol{a}}

\newcommand{\bsone}{\boldsymbol{1}}

\newcommand{\RR}{\mathbb{R}}

\newcommand{\NN}{\mathbb{N}}

\newcommand{\ZZ}{\mathbb{Z}}

\newcommand{\calP}{\mathcal{P}}

\allowdisplaybreaks

\title{Lower bounds on the $L_p$ discrepancy of digital NUT sequences}
\author{Ralph Kritzinger and Friedrich Pillichshammer\thanks{The authors are supported by the Austrian Science Fund (FWF): Project F5509-N26, which is a part of the Special Research Program "Quasi-Monte Carlo Methods: Theory and Applications".}}
\date{}

\begin{document}

\maketitle

\abstract{We study the $L_p$ discrepancy of digital NUT sequences which are an important sub-class of digital $(0,1)$-sequences in the sense of Niederreiter. The main result is a lower bound for certain sub-classes of digital NUT sequences. }

\centerline{\begin{minipage}[hc]{130mm}{
}
\end{minipage}}

 \allowdisplaybreaks

\section{Introduction}

For a set $\calP=\{x_0,\dots,x_{N-1}\}$ of $N$ points in $[0,1)$ the (non-normalized) {\it $L_p$ discrepancy} for $p\in[1,\infty]$ is defined as
$$L_p(\calP)=\|\Delta_{\calP}\|_{L_p([0,1])}=\left(\int_0^1 |\Delta_{\calP}(t)|^p \rd t\right)^{\frac{1}{p}}$$
(with the usual modification if $p=\infty$), where $$ \Delta_{\calP}(t)=\sum_{n=0}^{N-1}\bsone_{[0,t)}(x_n)-Nt \ \ \ \mbox{for $t \in [0,1]$} $$ 
is the (non-normalized) {\it discrepancy function} of $\calP$. 

We denote by $\NN$ the set of positive integers and define $\NN_0=\NN \cup \{0\}$. Let $X=(x_n)_{n \ge 0}$ be an infinite sequence in $[0,1)$ and, for $N \in \NN$, let $X_N=\{x_0,x_1,\ldots,x_{N-1}\}$ denote the set consisting of the first $N$ elements of $X$. It is well known that for all $p\in [1,\infty)$ we have $$L_p(X_N)\gtrsim \sqrt{\log N} \ \ \ \mbox{$\infty$-often}$$ and $$L_{\infty}(X_N)\gtrsim \log N \ \ \ \mbox{$\infty$-often.}$$ (For functions $f,g:\NN \rightarrow \RR^+$, we write $g(N) \lesssim f(N)$ or $g(N) \gtrsim f(N)$, if there exists a positive constant $C$ that is independent of $N$ such that $g(N) \le C f(N)$ or $g(N) \geq C f(N)$, respectively.) The lower estimate for finite $p$ was first shown by  Pro{\u\i}nov~\cite{pro86} (see also \cite{DP14a}) based on famous results of Roth~\cite{Roth2} and
Schmidt~\cite{schX} for finite point sets in dimension two. Using the method of  Pro{\u\i}nov in conjunction with a result of Hal\'{a}sz~\cite{hala} for finite point sets in dimension two the lower bound follows also for the $L_1$-discrepancy. The estimate for $p=\infty$ was first shown by Schmidt~\cite{Schm72distrib} in 1972 (see also \cite{be1982,lar2014,wagn}).\\

In this paper we investigate the $L_p$ discrepancy of {\it digital $(0,1)$-sequences}. Since we only deal with digital sequences over $\ZZ_2$ and in dimension 1 we restrict the necessary definitions to this case. For the general setting we refer to \cite{DP10,Nied87,Nied92}.

Let $\ZZ_2$ be the finite field of order 2, which we identify with the set $\{0,1\}$ equipped with arithmetic operations modulo 2. For the generation of a digital $(0,1)$ sequence $(x_n)_{n \ge 0}$ over $\ZZ_2$ we require an infinite matrix $C=(c_{i,j})_{i,j \ge 1}$ over $\ZZ_2$ with the following property\footnote{A further technical condition which is sometimes required, see \cite[p.72, (S6)]{Nied92}, is that for each $j \ge 1$ the sequence $(c_{i,j})_{i \ge 1}$ becomes eventually zero. Otherwise it could happen that one or more elements of the digital $(0,1)$-sequence are 1 and therefore do not belong to $[0,1)$.}: for every $n \in \NN$ the left upper $n \times n$ submatrix $(c_{i,j})_{i,j=1}^n$ has full rank. In order to construct the $n^{{\rm th}}$ element $x_n$ for $n \in \NN_0$ compute the base 2 expansion $n=n_0+n_1 2+n_2 2^2+\cdots$ (which is actually finite), set $\vec{n}:=(n_0,n_1,n_2,\ldots)^{\top} \in \ZZ_2^{\infty}$ and compute the matrix vector product $$C \vec{n}=:(y_1^{(n)},y_2^{(n)},y_3^{(n)},\ldots)^{\top} \in \ZZ_2^{\infty}$$ over $\ZZ_2$. Finally, set  $$x_n:=\frac{y_1^{(n)}}{2}+\frac{y_2^{(n)}}{2^2}+\frac{y_3^{(n)}}{2^3}+\cdots.$$
We denote the digital $(0,1)$-sequence\footnote{In the general notation, the $1$ refers to the dimension and the $0$ refers to the full rank condition of the generator matrix $C$.} constructed in this way by $X^C$.\\

An important sub-class of digital $(0,1)$-sequences which is studied in many papers (initiated by Faure~\cite{Fau05}) are so-called {\it digital NUT sequences} whose generator matrices are of non-singular upper triangular (NUT) form
\begin{equation} \label{matrixb} C=
\begin{pmatrix}
1 & c_{1,2} & c_{1,3}  &   \cdots\\
0 & 1 & c_{2,3}   &  \cdots\\
0 & 0 & 1   & \cdots\\ 
 \vdots & \vdots & \vdots & \ddots
\end{pmatrix}.
\end{equation}

For example, if $C=I$ is the identity matrix, then the corresponding digital NUT sequence is the {\it van der Corput sequence} in base 2. For information about digital NUT sequences and the van der Corput sequence see the survey \cite{FKP15} and the references therein.\\

For digital NUT sequences $X^C$ it is known (see \cite[Theorem~1]{DLP}), that $$L_2(X_N^C) \le L_2(X_N^I) \le \left(\left(\frac{\log N}{6 \log 2}\right)^2 + O(\log N) \right)^{1/2}$$ and for general $p \ge 1$ it is known (see \cite[Theorem~2]{pil04}) that 
\begin{equation}\label{uglkettedisc}
L_p(X_N^C) \le L_{\infty}(X_N^C) \le L_{\infty}(X_N^I) \le \frac{\log N}{3 \log 2}+1.
\end{equation}
Note that according to the lower bound of Schmidt the upper bound for the $L_{\infty}$ discrepancy in \eqref{uglkettedisc} is optimal in the order of magnitude in $N$. This is not the case for finite~$p$.

Concerning lower bounds on the $L_p$ discrepancy of digital NUT sequences very few is known and only for very special cases. For the van der Corput sequence we have for all $p \in [1,\infty)$ 
\begin{equation}\label{lpvdclsup}
\limsup_{N \rightarrow \infty} \frac{L_p(X_N^I)}{\log N}=\frac{1}{6 \log 2}
\end{equation}
and hence $L_p(X_N^I) \gtrsim \log N$ for infinitely many $N \in \NN$; see \cite[Corollary~1]{pil04}. 

For the so-called {\it upper-$\bsone$-sequence} $X^U$, which is generated by the matrix 
\begin{equation} \label{Umatrix} U=
\begin{pmatrix}
1 & 1 & 1  &   \cdots\\
0 & 1 & 1  &   \cdots\\
0 & 0 & 1  &  \cdots\\ 
\vdots & \vdots & \vdots & \ddots
\end{pmatrix}
\end{equation}
it is known that for every $p\ge 1$ we have $L_p(X_N^U) \ge \frac{\log N}{20 \log 2}+O(1)$ infinitely often; see \cite{DLP}.\\

In \cite{DLP} the authors study the $L_p$ discrepancy of $X^C$ for special types of NUT matrices $C$ of the form
\begin{eqnarray}\label{c2}
C=\left(
\begin{array}{ccc}
\bsa_1 &           &            \\
0         & \bsa_2 &            \\
0         & 0         & \bsa_3  \\
\multicolumn{3}{c}\dotfill
\end{array}\right)
\end{eqnarray}
with $$\bsa_i=(1,0,0,\ldots) \;\;\;\mbox{ or }  
\;\;\;\bsa_i=(1,1,1,\ldots)\;\;\;\mbox{ for } i \in \NN.$$ Note that these NUT sequences comprise the van der Corput sequence and the upper-$\bsone$-sequence $X^U$ as special cases. For $m \in \NN$ let $h(m)$ denote the number of $(1,0,0,\ldots)$ rows among the first $m$ rows of $C$. For example, $h(m)=m$ in case of the van der Corput sequence and $h(m)=0$ in case of the sequence $X^U$. Then it follows from \cite[Lemma~4]{DLP} that for every $m \in \NN$ there exists an integer $N \in [2^m,2^{m+1})$ such that $L_1(X_N^C) \gtrsim (m+h(m)^2)^{1/2}$. This implies that if $h(m) \gtrsim m$ we have for every $p \ge 1$ $$L_p(X_N^C) \gtrsim \log N \ \ \ \mbox{$\infty$-often.}$$

In general, however, it is a very difficult task to give precise lower bounds on the $L_p$ discrepancy of digital NUT sequences. We strongly conjecture the following:

\begin{conjecture}\label{con1}
For every digital NUT sequence $X^C$ we have 
\begin{equation}\label{eqcon1}
L_p(X_N^C) \gtrsim \log N \ \ \ \mbox{$\infty$-often.}
\end{equation} 
\end{conjecture}

Note that for every digital NUT sequence and for every $p\ge 1$ we have 
\begin{equation}\label{disckettes2}
L_p(X_N^C) \le L_{\infty}(X_N^C) \le L_{\infty}(X_N^I) \le s_2(N),
\end{equation}
where $s_2: \NN \rightarrow \NN$ denotes the binary sum-of-digits function which is defined as $s_2(N)=N_0+N_1+\cdots+N_m$ whenever $N$ has binary expansion $N=N_0+N_1 2 + \cdots+N_m 2^m$. The very last estimate in \eqref{disckettes2} follows from the proof of \cite[Theorem~3.5 in Chapter~2]{kuinie}. 

\begin{remark}
The result in \eqref{disckettes2} can be generalized and improved in the following sense:
For every $p\in [1,\infty]$ and for {\bf every} digital $(0,1)$-sequence $X^C$ we have
$$L_p(X_N^C)\le c_p s_2(N)\ \ \ \mbox{for all $N \in \NN$,}$$
where $$c_p=\left\{ 
\begin{array}{ll}
1/\sqrt{3}=0.5773\ldots & \mbox{ if $p\in [1,2]$,}\\
1 & \mbox{ if $p \in (2,\infty]$.} 
\end{array}
\right.$$ We omit the proof.
\end{remark}

The sum-of-digits function is very fluctuating. For example we have $s_2(2^m)=1$, but $s_2(2^m-1)=m$. In any case we have $s_2(N)\le \frac{\log N}{\log 2} +1$. 

\begin{remark}\rm
The inequalities in \eqref{disckettes2} shows that having only very few non-zero binary digits is a sufficient condition on $N \in \NN$ which guarantees that $X_N^C$ has very low $L_p$ discrepancy. For example we have $$L_p(X_{N}^C)\le 1 \ \ \ \mbox{for all $N$ of the form $N=2^m$}$$ or $$L_p(X_{N}^C)\lesssim \sqrt{\log N} \ \ \ \mbox{for all $N$ of the form $N=2^m+2^{\lfloor \sqrt{m}\rfloor -1} -1$}$$ or $$ L_p(X_N^C)\lesssim \log N \ \ \ \mbox{ for all $N \ge 2$.}$$ See Figure~\ref{fig_a} for a comparison for the van der Corput sequence.

However, the condition on $N$ of having very few non-zero binary digits is not a necessary one for low discrepancy.  For example, consider $N$ of the form $N=2^m-1$. Then we have $s_2(N)=m=\lfloor \frac{\log N}{\log 2}+1\rfloor$ but: since the discrepancy of $X_N^C$ and of $X_{N+1}^C$ differ at most by 1 and since $L_p(X_{N+1}^C)=L_p(X_{2^m}^C) \le 1$ we obtain $L_p(X_N^C) \lesssim 1$. Hence, while $s_2(N)$ is very large, the discrepancy $L_p(X_N^C)$ is low.  

But in any case: the only possible candidates of $N$ that satisfy \eqref{eqcon1} are required to have $s_2(N) \gtrsim \log N$. 
\end{remark}

\begin{figure}[ht!]
\begin{center}
\includegraphics[angle=0,width=120mm]{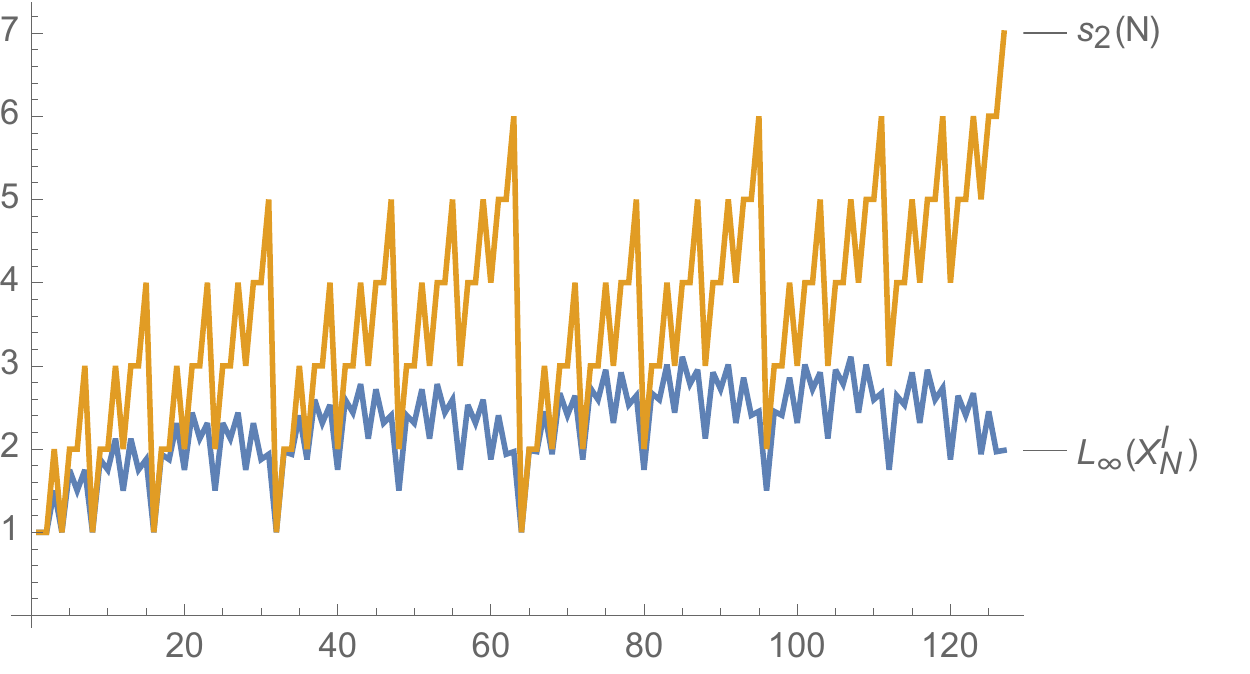}
\caption{The $L_{\infty}$ discrepancy of the van der Corput sequence $L_{\infty}(X_N^I)$ and the \newline binary sum-of-digits function $s_2(N)$ for $N=1,2,\ldots,127$.} \label{fig_a}
\end{center}
\end{figure}

In Section~\ref{secLB} we provide a lower bound for $L_p(X_N^C)$ for special types of NUT matrices.

\section{Lower bound on $L_p(X_N^C)$}\label{secLB}

We study two sub-classes of NUT matrices. The first class has a certain band structure. More detailled, the considered matrices are of the form $C(\alpha)=(c_{i,j})_{i,j \ge 1}$ where, for fixed $\alpha \in \NN$, $$c_{i,j}=\left\{ 
\begin{array}{ll}
1 & \mbox{if $i\le j < i+\alpha$,}\\
0 & \mbox{in all other cases.}
\end{array}\right.$$
For example, if $\alpha=1$, we obtain the identity matrix, i.e., $C(1)=I$. 

\begin{theorem}\label{thmlbd}
For all $\alpha \in \NN$ and $p\in [1,\infty]$ we have
$$L_p(X_N^{C(\alpha)}) \ge \frac{2^{\alpha-1}}{2^{2\alpha}-1}\frac{\log N}{2 \alpha \log 2}+O_{\alpha}(1) \ \ \ \mbox{$\infty$-often.}$$
The bound above is satisfied for $N$ of the form $$N=\sum_{\ell=1}^r 2^{2 \alpha (r-\ell)}=\frac{2^{2 \alpha r}-1}{2^{2 \alpha}-1}\ \ \ \mbox{for arbitrary $r \in \NN$.}$$
\end{theorem}

\begin{remark}\rm
\begin{enumerate}
\item Following all the details in the proof the constant hidden in $O_\alpha(1)$ can be computed exactly.
\item For $\alpha=1$ we have $C(1)=I$ and hence the resulting NUT sequence is the van der Corput sequence. Theorem~\ref{thmlbd} gives $$L_p(X_N^I) \ge \frac{\log N}{6 \log 2}+O(1) \ \ \ \mbox{$\infty$-often.}$$ This matches the corresponding value in \eqref{lpvdclsup}.
\end{enumerate}
\end{remark}

We also study NUT matrices which have the same entries in each column above the diagonal; i.e. we deal with matrices of the form 
\begin{equation} \label{Umatrix1} C(\bsa)=
\begin{pmatrix}
1 & a_1 & a_2  & a_3 &  \cdots\\
0 & 1 & a_2  & a_3 &  \cdots\\
0 & 0 & 1  & a_3 & \cdots\\ 
0 & 0 & 0  & 1 & \cdots \\
\vdots & \vdots & \vdots & \vdots & \ddots
\end{pmatrix},
\end{equation}
where $\bsa=(a_1,a_2,\dots)\in \ZZ_2^{\NN}$ is chosen arbitrarily. We set $l_0(m):=\#\{i\in\{1,\dots,m\}: a_i=0\}$ and $l_1(m):=\#\{i\in\{1,\dots,m\}: a_i=1\}$.
For $m\geq 2$ let further $d_0(m)$ be the minimal distance of consecutive zeroes and $d_1(m)$ be the minimal distance of consecutive ones in the string $(a_1,\dots,a_m)$, i.e. for $\ell\in\{0,1\}$ we define
$$ d_{\ell}(m):=\min_{1\leq n \leq m-1} \left\{\exists i\in\{1,\dots,m-n\}: a_{i}=a_{i+n}=\ell, a_{i+1}=\dots=a_{i+n-1}\neq \ell\right\}. $$
\begin{theorem}\label{thmlbd2}
  Let $m\geq 2$. For all $p\in [1,\infty]$ and for $N_{\bsa}=1+\sum_{i=1}^{m-1} 2^{i}(1-a_i)+2^m \ \ \ \mbox{for arbitrary $m \in \NN$}$ we have
\begin{equation} \label{claim1} L_p(X_{N_{\bsa}}^{C(\bsa)}) \ge \frac{1}{3}l_0(m)+O(1) \end{equation} 
if $d_0(m)\geq 2$, and
\begin{equation} \label{claim2} L_p(X_{N_{\bsa}}^{C(\bsa)}) \ge \frac{1}{3}l_1(m)+O(1) \end{equation} 
if $d_1(m)\geq 2$.
\end{theorem}

\begin{corollary} \label{coro1}
The first $N$ elements of a NUT-sequence generated by a matrix of the form $C(\bsa)$ satisfy
	  \begin{equation} L_p(X_N^{C(\bsa)}) \ge c \log{N} \ \ \ \mbox{for infinitely many $N$} \label{baddiscr} \end{equation} 
		for some constant $c>0$ if $l_1(m)\geq c_1 m$ for some $c_1>0$ and $d_1(m)\geq 2$ for all $m\geq 2$ or
		if $l_0(m)\geq c_2 m$ for some $c_2>0$ and $d_0(m)\geq 2$ for all $m\geq 2$. 
\end{corollary}

One example for a generator matrix satisfying the hypotheses of Corollary~\eqref{coro1} is
$$C(\bsa)=
\begin{pmatrix}
1 & 0 & 1  & 0 & 1 & \cdots\\
0 & 1 & 1  & 0 & 1 & \cdots\\
0 & 0 & 1  & 0 & 1 & \cdots\\ 
0 & 0 & 0  & 1 & 1 & \cdots \\
0 & 0 & 0  & 0 & 1 & \cdots\\
\vdots & \vdots & \vdots & \vdots & \vdots & \ddots
\end{pmatrix}.
$$

\section{The proofs}

The following auxiliary result will be the main tool of our proofs. 

\begin{lemma}\label{le17}
For every NUT digital sequence $X^C$ and every $N \in \NN$ of the form $N=2^{n_1}+ 2^{n_2}+\cdots + 2^{n_r}$ with $n_1>n_2> \ldots >n_r$ and $r \in \NN$ we have 
$$\int_0^1 \Delta_{X_N^C}(t) \rd t = \sum_{i=2}^r\sigma_{r,n_i+1}-\sum_{k=2}^r \sum_{j=n_k +1}^{n_{k-1}} \frac{\sigma_{r,j}}{2^j} \sum_{i=k}^r 2^{n_i}+O(1),$$ where the $\sigma_{r,j}$ are given by the following matrix-vector product over $\ZZ_2$:
$$\left(
\begin{array}{l}
\sigma_{r,n_r+1}\\
\vdots\\
\vdots\\
\vdots\\
\vdots\\
\sigma_{r,n_1+1}
\end{array}\right)=
\left(
\begin{array}{llll}
c_{n_r+1,n_r+1} &  \ldots  & \ldots & c_{n_r+1,n_1+1} \\
\multicolumn{4}{c}\dotfill\\
\multicolumn{4}{c}\dotfill\\
\multicolumn{4}{c}\dotfill\\
\multicolumn{4}{c}\dotfill\\
0 \ldots &  \ldots  & 0 & c_{n_1+1,n_1+1} \\
\end{array}
\right)
\left(
\begin{array}{l}
0 \\
\vdots \\
0 \\
\hline
1 \\
0 \\
\vdots\\
0 \\
\hline
\vdots\\
\hline
1  \\
0 \\
\vdots\\
0  \\
\hline
1 
\end{array}\right),
$$  
where the digits 1 in the latter vector are placed at positions $n_l-n_r+1$ for $l\in\{1,\dots,r-1\}$.
\end{lemma}

\begin{proof}
Let $X^C=(x_n)_{n \ge 0}$ be the NUT digital sequence which is generated by the $\NN\times\NN$ matrix $C$. Let $N \in \NN$ be of the form $$N=2^{n_1}+2^{n_2}+\cdots+2^{n_r},$$ where $n_1 > \ldots > n_r$. 
For $i=1,\ldots,r$ consider $$\calP_i=\{x_{2^{n_1}+\cdots+2^{n_{i-1}}},\ldots,x_{2^{n_1}+\cdots+2^{n_{i-1}}+2^{n_i}-1}\},$$ where for $i=1$ we define $2^{n_1}+\cdots+2^{n_{i-1}}=0$. Every 
\begin{equation}\label{nrangePi}
n \in \{2^{n_1}+\cdots+2^{n_{i-1}},\ldots,2^{n_1}+\cdots+2^{n_{i-1}}+2^{n_i}-1\}
\end{equation}
can be written as $$n=2^{n_1}+\cdots+2^{n_{i-1}}+a=2^{n_{i-1}} l_i+a,$$ where $a \in \{0,1,\ldots,2^{n_i}-1\}$ and $$l_i=\left\{
\begin{array}{ll}
0 & \mbox{ if $i=1$,} \\
1 & \mbox{ if $i=2$,} \\
1+2^{n_{i-2}-n_{i-1}}+\cdots +2^{n_1-n_{i-1}} & \mbox{ if $i> 2$.}
\end{array}
\right.$$

For fixed $i=1,\ldots,r$ we decompose the matrix $C$ in the form 
$$
 \left( \begin{array}{ccc}
           & \vline &  \\
  C^{(n_i \times n_i)} & \vline & D^{(n_i \times \mathbb{N})} \\
           & \vline &   \\ \hline
    & \vline & \\
 0^{(\mathbb{N} \times n_i)}  & \vline &  F^{(\mathbb{N} \times \mathbb{N})} \\
   &   \vline       &
\end{array} \right) \in \mathbb{Z}_2^{\mathbb{N} \times \mathbb{N}},
$$  where $C^{(n_i \times n_i)}$ is the left upper $n_i\times n_i$ sub-matrix of $C$. To $n$ in~\eqref{nrangePi} we associate $$\vec{n}=(a_0,a_1,\ldots,a_{n_i-1},\ell_0,\ell_1,\ell_2,\ldots)^{\top}=:{\vec{a} \choose \vec{l}_i},$$ where $a_0,\ldots,a_{n_i-1}$ are the binary digits of $a$ and $\ell_0,\ell_1,\ell_2,\ldots$ are the binary digits of $l_i$. With this notation for $n$ in the range \eqref{nrangePi} we have
$$C \vec{n}=\left( \begin{array}{c} C^{(n_i \times n_i)}  \vec{a} \\ 0 \\ 0 \\ \vdots \end{array} \right)  + \left( \begin{array}{c}
             \\
  D^{(n_i \times \mathbb{N})}  \\
          \\ \hline  \\
 F^{(\mathbb{N} \times \mathbb{N})}  \\
   \end{array} \right) \vec{l}_i.
	$$ 
This shows that the point set $\calP_i$ is a digitally shifted digital net with generating matrix $C^{(n_i \times n_i)}$ and with digital shift vector
\begin{equation}\label{dig_shift}
\vec{\sigma}_{i}=(\sigma_{i,1},\sigma_{i,2},\ldots)^{\top}:=\left( \begin{array}{c}
             \\
  D^{(n_i \times \mathbb{N})}  \\
          \\ \hline  \\
 F^{(\mathbb{N} \times \mathbb{N})}  \\
   \end{array} \right) \vec{l}_i .
\end{equation}
Since $F^{(\NN \times \NN)}$ is also a NUT matrix we find that the shift is of the form $$\vec{\sigma}_i=(\sigma_{i,1},\sigma_{i,2},\ldots,\sigma_{i,n_1+1},0,0,\ldots)^{\top}\in \ZZ_2^{\infty}$$

Note that the matrix $C^{(n_i \times n_i)}$ has full rank, as $X^C$ is a NUT digital sequence. Hence the shifted digital net $\calP_i$ can be written as the set of points
$$  \calP_i=\left\{\frac{b_1}{2}+\dots+\frac{b_{n_i}}{2^{n_i}}+\sum_{j=1}^{\infty}\frac{\sigma_{i,n_i+j}}{2^{n_i+j}}:a_0,\dots,a_{n_i-1} \in\{0,1\}\right\}, $$
where $b_k=c_{k,1}a_0\oplus \dots \oplus c_{k,n_i}a_{n_i-1}\oplus \sigma_{i,k}$ for $1\leq k\leq n_i$. Here and in the following $\oplus$ denotes addition in $\ZZ_2$.

We emphasize that $\sigma_{i,1},\dots,\sigma_{i,n_i}$ do not depend on the $a_i$'s, whereas the components $\sigma_{i,j}$ for $j\geq n_i+1$ may do so. Therefore we can also write
$$ \calP_i=\left\{\frac{k_i}{2^{n_i}}+\delta_{i}:k_i \in\{0,1,\dots,2^{n_i}-1\}\right\}, $$
where $\delta_{i}=\sum_{j=1}^{n_1-n_i+1}\frac{\sigma_{i,n_i+j}}{2^{n_i+j}}$ for $i>1$ and $\delta_1=0$.

We have the following decomposition of $X_N^C$: $$X_N^C=\bigcup_{i=1}^r \calP_i.$$

Therefore and from the fact that
$$ \int_0^1 \Delta_{\calP_i}(t)\, \rd t=\sum_{z \in \calP_i}\left(\frac12-z\right) $$
 we obtain
\begin{eqnarray*}
\int_0^1 \Delta_{X_N^C}(t)\, \rd t & = & \sum_{i=1}^r \int_0^1 \Delta_{\calP_i}(t)\, \rd t
 =  \sum_{i=1}^r \sum_{\ell=0}^{2^{n_i}-1}\left(\frac12-\left(\frac{\ell}{2^{n_i}}+\delta_i\right)\right)\\
& = & \sum_{i=1}^r \left(\frac12-2^{n_i}\delta_i\right) 
= \frac{r}{2}- \sum_{i=1}^r2^{n_i}\delta_i, 
\end{eqnarray*}
 where
 $$2^{n_i} \delta_i=\left\{ 
\begin{array}{ll}
0 & \mbox{ if } i=1,\\
\frac{\sigma_{i,n_i+1}}{2}+\frac{\sigma_{i,n_i+2}}{2^2}+\frac{\sigma_{i,n_i+3}}{2^3}+\cdots +\frac{\sigma_{i,n_1+1}}{2^{n_1-n_i+1}}& \mbox{ if } i>1
\end{array}\right.
$$ and, for $k \ge 1$, 
\begin{eqnarray*}
\sigma_{i,n_i+k} & = & \bigoplus_{j=0}^{n_1-n_i-k+1} c_{n_i+k,n_i+k+j} a_{n_i+k-1+j}\\ & = & c_{n_i+k,n_i+k} a_{n_i+k-1}+ c_{n_i+k,n_i+k+1} a_{n_i+k}+ \cdots + c_{n_i+k,n_1+1} a_{n_1}\pmod{2},
\end{eqnarray*}
where $a_{n_\ell}=1$ for $\ell=1,\ldots,i-1$ and all other $a_r$'s are zero. Note that $\sigma_{i,n_i+k}\in \ZZ_2$.


We have 
\begin{eqnarray*}
\sum_{i=1}^r2^{n_i}\delta_i & = & \sum_{i=2}^r \sum_{j = 1}^{n_1-n_i+1}\frac{\sigma_{i,n_i+j}}{2^j}.
\end{eqnarray*}
Observe that
$$\left(
\begin{array}{l}
\sigma_{i,n_i+1}\\
\ldots \\
\sigma_{i,n_{i-1}}\\
\sigma_{i,n_{i-1}+1}\\
\ldots\\
\ldots\\
\ldots\\
\ldots\\
\sigma_{i,n_1+1}
\end{array}\right)=
\left(
\begin{array}{llllll}
c_{n_i+1,n_i+1}  & \ldots & c_{n_i+1,n_{i-1}+1} & \ldots & \ldots  & c_{n_i+1,n_1+1} \\
\multicolumn{6}{c}\dotfill\\
\multicolumn{6}{c}\dotfill\\
0 \ldots &  \ldots & c_{n_{i-1}+1,n_{i-1}+1} & \ldots & \ldots   & c_{n_{i-1}+1,n_{1}+1} \\
\multicolumn{6}{c}\dotfill\\
\multicolumn{6}{c}\dotfill\\
\multicolumn{6}{c}\dotfill\\
\multicolumn{6}{c}\dotfill\\
0 \ldots & \multicolumn{4}{c}\dotfill   & c_{n_1+1,n_1+1} \\
\end{array}
\right)
\left(
\begin{array}{l}
0\\
\vdots\\
0\\
1\\
0\\
\vdots \\
0\\
\vdots\\
1
\end{array}\right),
$$ 
and
$$\left(
\begin{array}{l}
\sigma_{i,n_{i-1}+1}\\
\ldots\\
\ldots\\
\ldots\\
\ldots\\
\sigma_{i,n_1+1}
\end{array}\right)=
\left(
\begin{array}{llll}
c_{n_{i-1}+1,n_{i-1}+1} &  \ldots  & \ldots & c_{n_{i-1}+1,n_1+1} \\
\multicolumn{4}{c}\dotfill\\
\multicolumn{4}{c}\dotfill\\
\multicolumn{4}{c}\dotfill\\
\multicolumn{4}{c}\dotfill\\
0 \ldots &  \ldots  & 0 & c_{n_1+1,n_1+1} \\
\end{array}
\right)
\left(
\begin{array}{l}
1\\
0\\
\vdots \\
0\\
\vdots\\
1
\end{array}\right).
$$ 
Hence 
$$\left(
\begin{array}{l}
\sigma_{i-1,n_{i-1}+1}\\
\sigma_{i-1,n_{i-1}+2}\\
\ldots\\
\ldots\\
\ldots\\
\sigma_{i-1,n_1+1}
\end{array}\right)=
\left(
\begin{array}{llll}
c_{n_{i-1}+1,n_{i-1}+1} &  \ldots  & \ldots & c_{n_{i-1}+1,n_1+1} \\
\multicolumn{4}{c}\dotfill\\
\multicolumn{4}{c}\dotfill\\
\multicolumn{4}{c}\dotfill\\
\multicolumn{4}{c}\dotfill\\
0 \ldots &  \ldots  & 0 & c_{n_1+1,n_1+1} \\
\end{array}
\right)
\left(
\begin{array}{l}
0\\
0\\
\vdots \\
0\\
\vdots\\
1
\end{array}\right)=
\left(
\begin{array}{l}
\sigma_{i,n_{i-1}+1} \oplus 1\\
\sigma_{i,n_{i-1}+2}\\
\ldots\\
\ldots\\
\ldots\\
\sigma_{i,n_1+1}
\end{array}\right).
$$ 
This shows that we have $$\sigma_{i,k} =\left\{ 
\begin{array}{ll}
\sigma_{i-1,k} & \mbox{for $k=n_{i-1}+2,n_{i-1}+3,\ldots,n_1+1$,}\\
\sigma_{i-1,k} \oplus 1 & \mbox{for $k=n_{i-1}+1$.}\\
\end{array}\right.$$
From this we obtain for all $i\in\{2,3,\ldots,r\}$ that 
$$\sigma_{i,n_i+j} =\left\{ 
\begin{array}{ll}
\sigma_{r,n_i+j} & \mbox{for $j=2,3,\ldots,n_1-n_i+1$,}\\
\sigma_{r,n_i+j} \oplus 1 =1- \sigma_{r,n_i+j}& \mbox{for $j=1$.}\\
\end{array}\right.$$

Hence
\begin{eqnarray*}
\sum_{i=1}^r2^{n_i}\delta_i & = & \sum_{i=2}^r \frac{1-\sigma_{r,n_i+1}}{2}+ \sum_{i=2}^r \sum_{j = 2}^{n_1-n_i+1}\frac{\sigma_{r,n_i+j}}{2^j} 
= \frac{r-1}{2}- \sum_{i=2}^r\sigma_{r,n_i+1}+\sum_{i=2}^r \sum_{j = 1}^{n_1-n_i+1}\frac{\sigma_{r,n_i+j}}{2^j}.
\end{eqnarray*}
For the very last double sum we have
\begin{eqnarray*}
\sum_{i=2}^r \sum_{j = 1}^{n_1-n_i+1} \frac{\sigma_{r,n_i+j}}{2^j} & = & \sum_{i=2}^r \sum_{j=n_i+1}^{n_1+1} \frac{\sigma_{r,j}}{2^{j-n_i}}= \sum_{j=n_r +1}^{n_1+1} \frac{\sigma_{r,j}}{2^j} \sum_{\substack{i=2 \\ n_i \le j-1}}^r 2^{n_i}\\
& = & \sum_{k=2}^r \sum_{j=n_k +1}^{n_{k-1}} \frac{\sigma_{r,j}}{2^j} \sum_{\substack{i=2 \\ n_i \le j-1}}^r 2^{n_i} + \frac{\sigma_{r,n_1+1}}{2^{n_1+1}} \sum_{\substack{i=2 \\ n_i \le n_1}}^r 2^{n_i}\\
& = & \sum_{k=2}^r \sum_{j=n_k +1}^{n_{k-1}} \frac{\sigma_{r,j}}{2^j} \sum_{i=k}^r 2^{n_i}+ \frac{\sigma_{r,n_1+1}}{2^{n_1+1}} (N-2^{n_1}).
\end{eqnarray*}
Hence
$$\sum_{i=1}^r2^{n_i}\delta_i = \frac{r-1}{2}- \sum_{i=2}^r\sigma_{r,n_i+1}+\sum_{k=2}^r \sum_{j=n_k +1}^{n_{k-1}} \frac{\sigma_{r,j}}{2^j} \sum_{i=k}^r 2^{n_i}+ \frac{\sigma_{r,n_1+1}}{2^{n_1+1}} (N-2^{n_1}).$$ This gives
$$\int_0^1 \Delta_{X_N^C}(t) \rd t = \sum_{i=2}^r\sigma_{r,n_i+1}-\sum_{k=2}^r \sum_{j=n_k +1}^{n_{k-1}} \frac{\sigma_{r,j}}{2^j} \sum_{i=k}^r 2^{n_i}+O(1).$$
\end{proof}

Now we give the proof of Theorem~\ref{thmlbd}.\\

\begin{proof}
In order to simplify the notation we will write $C$ instead of $C(\alpha)$ in the following. For every $N$ and $p$ we have
\begin{equation}\label{lpbdint}
L_p(X_N^C) \ge L_1(X_N^C)= \|\Delta_{X_N^C}\|_{L_1([0,1))} \ge \left|\int_0^1 \Delta_{X_N^C}(t) \rd t\right|.
\end{equation}

Now choose $n_i=2 \alpha (r-i)$ for $i \in\{1,2,\ldots,r\}$, i.e. $$N=\sum_{i=1}^r 2^{2 \alpha (r-i)}=\frac{2^{2 \alpha r}-1}{2^{2 \alpha}-1}\ \ \mbox{ and hence }\ \ r=\frac{\log((2^{2 \alpha}-1)N +1)}{2 \alpha \log 2}.$$ We have $$\sum_{i=k}^r 2^{n_i}=\sum_{i=k}^r 2^{2 \alpha (r-i)}=\frac{2^{2 \alpha(r-k+1)}-1}{2^{2 \alpha}-1}.$$ Therefore
\begin{eqnarray*}
\sum_{k=2}^r \sum_{j=n_k +1}^{n_{k-1}} \frac{\sigma_{r,j}}{2^j} \sum_{i=k}^r 2^{n_i} & = & \sum_{k=2}^r \sum_{j=2\alpha(r-k) +1}^{2\alpha(r-k)+2\alpha} \frac{\sigma_{r,j}}{2^j} \frac{2^{2 \alpha(r-k+1)}-1}{2^{2 \alpha}-1}\\
& = & \frac{1}{2^{2 \alpha}-1} \sum_{k=2}^r \sum_{j=1}^{2\alpha} \frac{\sigma_{r,2 \alpha(r-k)+j}}{2^{2 \alpha(r-k)+j}} (2^{2 \alpha(r-k+1)}-1)\\
& = & \frac{2^{2 \alpha}}{2^{2 \alpha}-1} \sum_{k=2}^r \sum_{j=1}^{2\alpha} \frac{\sigma_{r,2 \alpha(r-k)+j}}{2^j}+O(1)\\
& = & \frac{2^{2 \alpha}}{2^{2 \alpha}-1} \sum_{\ell=0}^{r-2} \sum_{j=1}^{2\alpha} \frac{\sigma_{r,2 \alpha \ell+j}}{2^j}+O(1).
\end{eqnarray*}
Hence, using Lemma~\ref{le17}, we get 
$$\int_0^1 \Delta_{X_N^C}(t) \rd t = \sum_{\ell=0}^{r-2}\sigma_{r,2 \ell \alpha+1}-\frac{2^{2 \alpha}}{2^{2 \alpha}-1} \sum_{\ell=0}^{r-2} \sum_{j=1}^{2\alpha} \frac{\sigma_{r,2\alpha \ell+j}}{2^j}+O(1).$$
Now we have to determine the numbers $\sigma_{r,j}$. Observe that
$$\left(
\begin{array}{l}
\sigma_{r,1}\\
\vdots\\
\vdots\\
\vdots\\
\vdots\\
\sigma_{r,(2r-2)\alpha+1}
\end{array}\right)=
\left(
\begin{array}{llll}
c_{1,1} &  \ldots  & \ldots & c_{1,(2r-2)\alpha+1} \\
\multicolumn{4}{c}\dotfill\\
\multicolumn{4}{c}\dotfill\\
\multicolumn{4}{c}\dotfill\\
\multicolumn{4}{c}\dotfill\\

0 \ldots &  \ldots  & 0 & c_{(2r-2)\alpha+1,(2r-2)\alpha+1} \\
\end{array}
\right)
\left(
\begin{array}{l}
0 \\
\vdots \\
0 \\
\hline
1 \\
0 \\
\vdots\\
0 \\
\hline
\vdots\\
\hline
1  \\
0 \\
\vdots\\
0 \\
\hline
1 
\end{array}\right),
$$ 
where the digits 1 in the latter vector are placed at positions $l\alpha+1$ for $l\in\{2,\dots,2r-2\}$.
From the structure of the matrix we find that
\begin{eqnarray*}
\sigma_{r,1}=\ldots =\sigma_{r,\alpha+1} & = & 0\\
\sigma_{r,\alpha+2}=\ldots =\sigma_{r,2\alpha+1} & = & 1\\
\sigma_{r,2\alpha+2}=\ldots =\sigma_{r,3\alpha+1} & = & 0\\
\sigma_{r,3\alpha+2}=\ldots =\sigma_{r,4\alpha+1} & = & 1\\
\sigma_{r,4\alpha+2}=\ldots =\sigma_{r,5\alpha+1} & = & 0\\
\ldots \\
\sigma_{r,(2r-3)\alpha+2}=\ldots =\sigma_{r,(2r-2)\alpha+1} & = & 1
\end{eqnarray*}
and therefore
\begin{eqnarray*}
\sum_{\ell=0}^{r-2} \sum_{j=1}^{2\alpha} \frac{\sigma_{r,2\alpha \ell+j}}{2^j} & = & \frac{1}{2^{\alpha+2}}+\cdots+\frac{1}{2^{2 \alpha}}+\sum_{\ell=1}^{r-2}\left(\frac{1}{2}+\frac{1}{2^{\alpha+2}}+\cdots+\frac{1}{2^{2 \alpha}}\right)\\
& = & \frac{r-2}{2}+(r-1)  \frac{1}{2^{\alpha+2}} \frac{1-(1/2)^{\alpha-1}}{1/2}.
\end{eqnarray*}
Furthermore $$\sum_{\ell=0}^{r-2} \sigma_{r,2 \alpha \ell+1}=0+1+1+\cdots+1=r-2.$$ Putting all together we obtain
\begin{eqnarray*}
\int_0^1 \Delta_{X_N^C}(t) \rd t & = & r-2-\frac{2^{2 \alpha}}{2^{2 \alpha}-1} \left( \frac{r-2}{2}+(r-1)  \frac{1}{2^{\alpha+2}} \frac{1-(1/2)^{\alpha-1}}{1/2}\right)+O(1)\\
& = & r \frac{2^{\alpha-1}}{2^{2\alpha}-1}+O(1).
\end{eqnarray*}
Hence, using \eqref{lpbdint}, we get $$L_p(X_N^C)\ge \left|\int_0^1 \Delta_{X_N^C}(t) \rd t\right| = \frac{2^{\alpha-1}}{2^{2\alpha}-1}\frac{\log((2^{2 \alpha}-1)N +1)}{2 \alpha \log 2}+O(1).$$
\end{proof}

In the following, we give the proof of Theorem~\ref{thmlbd2}.\\

\begin{proof}
Note that in the case $n_r=0$ the numbers $\sigma_{r,j}$ appearing in Lemma~\ref{le17} can also be understood in the following way: Let $N=2^{n_1}+\sum_{i=0}^{n_1-1}N_i2^i=\sum_{i=1}^r 2^{n_i}$ with $n_1=m \in\NN$,
	$N_i\in\ZZ_2$ for $i\in\{0,\dots,n_1-1\}$ and $r=s_2(N)$.  Let $$\eta_j:=c_{j,j+1}N_j\oplus \dots\oplus c_{j,n_1}N_{n_1-1}\oplus c_{j,n_1+1}.$$
	Then we have for $j\in\{1,\dots,n_1+1\}$
	$$ \sigma_{r,j}=\begin{cases}
	            1 & \text{if $j=n_1+1$}, \\
							\eta_j\oplus 1 & \text{if $j=n_k+1$ for some $k\in\{2,\dots,r\}$}, \\
							\eta_j & \text{otherwise.}
	                \end{cases}$$
Now consider a matrix of the form $C(\bsa)$ and set $N_{\bsa}=2^{n_1}+\sum_{i=1}^{n_1-1}(1-a_i)2^i+1=\sum_{i=1}^r 2^{n_i}$, where $r=l_0(m)+2$. Then we have
$$ \eta_j=a_{j}N_j\oplus \dots \oplus a_{n_1-1}N_{n_1-1}\oplus a_{n_1}=a_{n_1}. $$ 
We observe that for $N_{\bsa}$ and $j\in\{1,\dots,n_1\}$ we have $\sigma_{r,j}=a_{n_1}\oplus 1$ if and only if
 $j=n_k+1$ for some $k\in\{2,\dots,r\}$, and $\sigma_{r,j}=a_{n_1}$ otherwise. Hence with Lemma 1 we find 
  $$\int_0^1 \Delta_{X_{N_{\bsa}}^{C(\bsa)}}(t) \rd t = (-1)^{a_{n_1}}\left(\frac{r}{2}-\frac12 \sum_{k=2}^r \frac{1}{2^{n_k}}\sum_{i=k+1}^r 2^{n_i}\right)+O(1).$$
	The fact that $d_0(m)\geq 2$ implies $n_i-n_k \leq 2(k-i)$ and further
	$$ \left|\int_0^1 \Delta_{X_{N_{\bsa}}^{C(\bsa)}}(t) \rd t \right| \geq \frac{r}{2}-\frac12 \sum_{k=2}^r \sum_{i=k+1}^r 2^{2(k-i)}+O(1)=\frac{r}{3}+O(1). $$
	This completes the proof of the first claim~\eqref{claim1}. To derive~\eqref{claim2} from~\eqref{claim1}, we show that changing the tuple $\bsa$ which defines the matrix to $\tilde{\bsa}=(1-a_i)_{i\geq 1}$ does not change the integral of $\Delta_{X_{N_{\bsa}}^{C(\bsa)}}$ much. We use the following argument: Let for $2^{n_1}\leq N \leq 2^{n_1+1}-1$ with $N=2^{n_1}+\sum_{i=0}^{n_1-1}N_i2^i=\sum_{i=1}^r 2^{n_i}$ 
	  $$ S(N):=\frac{r}{2}-\frac12 \sum_{k=2}^r \frac{1}{2^{n_k}}\sum_{i=k+1}^r 2^{n_i}. $$
		It is not hard to show that $S(N)=\frac12 \sum_{\ell=0}^{n_1-1}\left\|\frac{N}{2^{\ell+1}}\right\|+O(1),$ where $\|x\|$ denotes the distance of a real number $x$ to its nearest integer. For $N$ as defined above we define the integer $N':=2^{n_1}+\sum_{i=0}^{n_1-1}(1-N_i)2^i$
		and prove $S(N')=S(N)+O(1)$. This is the case, since
		\begin{align*}
		   S(N)-S(N')=&\sum_{\ell=0}^{m-1}\left\{\left\|\frac{N_r}{2}+\dots+\frac{N_0}{2^{r+1}}\right\|-\left\|\frac{1-N_r}{2}+\dots+\frac{1-N_0}{2^{r+1}}\right\|\right\}+O(1) \\
			   =& \sum_{\substack{\ell=0 \\ N_r=0}}^{m-1}(-2^{-r-1})+\sum_{\substack{\ell=0 \\ N_r=1}}^{m-1}2^{-r-1}+O(1)=\sum_{\ell=0}^{m-1}(2N_r-1)2^{-r-1}+O(1)
		\end{align*}
		and therefore $$ |S(N)-S(N')|\leq \sum_{r=0}^{m-1}2^{-r-1}+O(1)=O(1). $$
		This implies inequality~\eqref{claim2}.
\end{proof}

%
%

\noindent{\bf Author's Addresses:}

\noindent Ralph Kritzinger and Friedrich Pillichshammer, Institut f\"{u}r Finanzmathematik, Johannes Kepler Universit\"{a}t Linz, Altenbergerstra{\ss}e 69, A-4040 Linz, Austria. Email: ralph.kritzinger(at)jku.at, friedrich.pillichshammer(at)jku.at
\end{document}